\documentclass[a4paper,12pt]{article}
\usepackage[top=2.5cm,bottom=2.5cm,left=2.5cm,right=2.5cm]{geometry}
\usepackage{cite,amsmath}
    \usepackage[margin=1cm,%
                font=small,%
                format=hang,%
                labelsep=period,%
                labelfont=bf]{caption}
\usepackage[algoruled,linesnumbered]{algorithm2e}
\usepackage{algorithmic}

\usepackage{amsfonts,epsf,here,graphicx}
\usepackage{latexsym,multirow,rotating}
\usepackage{pgf,tikz,subfigure}
\usepackage{epstopdf}

\newtheorem{theorem}{\bf Theorem}[section]
\newtheorem{corollary}[theorem]{\bf Corollary}
\newtheorem{lemma}[theorem]{\bf Lemma}
\newtheorem{proposition}[theorem]{\bf Proposition}

\newcommand{\qed}{\hfill $\square$ \bigskip}

\begin{document}

\baselineskip=0.30in
\vspace*{60mm}

\begin{center}
{\LARGE \bf Computing the Mostar index in networks with applications to molecular graphs}
\bigskip \bigskip

{\large \bf Niko Tratnik
}
\bigskip\bigskip

\baselineskip=0.20in

\textit{Faculty of Natural Sciences and Mathematics, University of Maribor, Slovenia} \\
{\tt niko.tratnik@um.si, niko.tratnik@gmail.com}

\bigskip\medskip

(Received \today)

\end{center}

\noindent
\begin{center} {\bf Abstract} \end{center}
Recently, a bond-additive topological descriptor, named as the Mostar index, has been introduced as a measure of peripherality in networks. For a connected graph $G$, the Mostar index is defined as $Mo(G) = \sum_{e=uv \in E(G)} |n_u(e) - n_v(e)|$, where for an edge $e=uv$ we denote by $n_u(e)$ the number of vertices of $G$ that are closer to $u$ than to $v$ and by $n_v(e)$ the number of vertices of $G$ that are closer to $v$ than to $u$. In this paper, we generalize the definition of the Mostar index to weighted graphs and prove that the Mostar index of a weighted graph can be computed in terms of Mostar indices of weighted quotient graphs. As a consequence, we show that the Mostar index of a benzenoid system can be computed in sub-linear time with respect to the number of vertices. Finally, our method is applied to some benzenoid systems and to a fullerene patch.
\vspace{3mm}\noindent


\baselineskip=0.30in



\section{Introduction}

In the last decades, many numerical quantities of graphs have been introduced and extensively studied in order to describe various structural properties. Such graph invariants are most commonly referred to as topological descriptors or topological indices and are often defined by using degrees of vertices, distances between vertices, eigenvalues, symmetries, and many other properties of graphs. In mathematical chemistry and in chemoinformatics, they are used for the development of quantitative structure-activity relationships (QSAR) and quantitative structure-property relationships (QSPR) in which some properties of compounds are correlated with their chemical structure \cite{todeschini}. Therefore, such descriptors are also called \textit{molecular descriptors} and can be applied in the process of finding new compounds with desired properties \textit{in silico} instead of \textit{in vitro}, which can save time and money. However, in recent years topological descriptors have found enormous applications in a rapidly growing research of complex networks \cite{dehmer,estrada}, which include communications networks, social networks, biological networks, and many others. In such networks, these descriptors are used as centrality measures, bipartivity measures, irregularity measures, connectivity measures, etc. For an example, an important task of centrality measures is to find the actors with a crucial role within the network.

One of the oldest molecular descriptors is the well known \textit{Wiener index}, which is, for a a connected graph $G$, defined as 
$$W(G) = \sum_{\lbrace u,v \rbrace \subseteq V(G)}d_G(u,v),$$
where $d_G(u,v)$ represents the distance between vertices $u$ and $v$ in $G$. It was introduced by H.\ Wiener in 1947 \cite{wiener} in order to calculate the boiling points of paraffins. Until known, it has found various applications in chemistry and in network theory. Moreover, many mathematical result are known for this index \cite{knor1}. However, it turns out that for every tree $T$ the Wiener index can be computed as the sum of edge contributions. More precisely,
$$W(T) = \sum_{e =uv \in E(T)}n_u(e)n_v(e),$$
where for an edge $e=uv$ we denote by $n_u(e)$ the number of vertices of $T$ that are closer to $u$ than to $v$ and by $n_v(e)$ the number of vertices of $T$ that are closer to $v$ than to $u$. Therefore, I.\ Gutman \cite{gut_sz} introduced the \textit{Szeged index} for any connected graph $G$ as

$$Sz(G) = \sum_{e=uv \in E(G)}n_u(e)n_v(e).$$
The Szeged index was also a great success and different variations of this index appeared in the literature, for example the edge-Szeged index, the revised Szeged index, the weighted Szeged index, etc. For some recent research on the Szeged index see \cite{arock1,ji,klavzar_li}.

Very recently, another bond-additive topological index, named as the Mostar index, has been introduced \cite{mostar}. For any connected graph $G$, the \textit{Mostar index} of $G$, denoted as $Mo(G)$, is defined
as

$$Mo(G) = \sum_{e=uv \in E(G)} |n_u(e) - n_v(e)|.$$
This index measures peripherality of individual edges and then sums the contributions of all edges into a global measure of peripherality for a given graph. In the same paper, a simple cut method for computing the Mostar index of benzenoid systems was also presented. Note that the cut method is a very useful tool for the computation of such topological descriptors \cite{klavzar-2015}.

In the present paper, we first expand the definition of the Mostar index to weighted graphs, which are also sometimes called networks. Later on, we prove that the Mostar index of a weighted graph can be computed as the sum of Mostar indices of weighted quotient graphs obtained by a partition of the edge set that is coarser than the $\Theta^*$-partition. Such methods were recently developed also for other distance-based molecular descriptors, see \cite{brez-trat,cre-trat,klavzar-2016,li}. Moreover, we show that the Mostar index of a benzenoid system can be computed in sub-linear time with respect to the number of vertices. Finally, our method is used to compute the Mostar index for some benzenoid systems and for a fullerene patch.

\section{Preliminaries}

Unless stated otherwise, the graphs considered in this paper are simple, finite, and connected.  For a graph $G$, we denote by $V(G)$ the set of vertices of $G$ and by $E(G)$ the set of its edges. Moreover, $d_G(u,v)$ is the usual shortest-path distance between vertices $u, v\in V(G)$. 
\smallskip

\noindent
Let $G$ be a graph and $e=uv$ an edge of $G$. Throughout the paper we will use the following notation:
$$N_u(e|G) = \lbrace x \in V(G) \ | \ d_G(x,u) < d_G(x,v) \rbrace, $$
$$N_v(e|G) = \lbrace x \in V(G) \ | \ d_G(x,v) < d_G(x,u) \rbrace, $$

\noindent
Let $\mathbb{R}_0^+ = [0,\infty)$. If $w: V(G) \rightarrow \mathbb{R}_0^+$ and $w': E(G) \rightarrow \mathbb{R}_0^+$ are give weights, then $(G,w,w')$ is called a \textit{vertex-edge-weighted graph} or shortly just a \textit{weighted graph}. For any $e=uv \in E(G)$ we define:
$$n_u(e|(G,w)) = \sum_{x \in N_u(e|G)} w(x), \quad n_v(e|(G,w)) = \sum_{x \in N_v(e|G)} w(x).$$

\noindent
We now introduce the Mostar index of $(G,w,w')$ as
$$Mo(G,w,w') = \sum_{e =uv \in E(G)} w'(e) \left| n_u(e|(G,w)) - n_v(e|(G,w)) \right|.$$
Obviously, for $w,w' \equiv 1$ this is exactly the Mostar index of $G$.
\smallskip

\noindent
Let $e = xy$ and $f = ab$ be two edges of a graph $G$. If $$d_G(x,a) + d_G(y,b) \neq d_G(x,b) + d_G(y,a),$$
we say that $e$ and $f$ are in relation $\Theta$ (also known as Djokovi\' c-Winkler relation) and write $e \Theta f$. Note that in some graphs this relation is not transitive (for example in odd cycles), although it is always reflexive and symmetric. As a consequence, we often consider the smallest transitive relation that contains relation $\Theta$ (i.e.\ the transitive closure of $\Theta$) and denote it by $\Theta^*$. It is known that in a \textit{partial cube}, which is defined as an isometric subgraph of some hypercube, relation $\Theta$ is always transitive, so $\Theta = \Theta^*$. Moreover, the class of partial cubes contains many interesting molecular graphs (for example benzenoid systems and phenylenes). For more information on partial cubes and relation $\Theta$ see \cite{klavzar-book}.
\smallskip

\noindent
Let $ \mathcal{E} = \lbrace E_1, \ldots, E_t \rbrace$ be the $\Theta^*$-partition of the edge set $E(G)$ and $ \mathcal{F} = \lbrace F_1, \ldots, F_k \rbrace$ an arbitrary partition of $E(G)$. If every element of $\mathcal{E}$ is a subset of some element of $\mathcal{F}$, we say that $\mathcal{F}$ is \textit{coarser} than $\mathcal{E}$. In such a case $\mathcal{F}$ will be shortly called a \textit{c-partition}. 
\smallskip

\noindent
Suppose $G$ is a graph and $F \subseteq E(G)$ is some subset of its edges. The \textit{quotient graph} $G / F$ is defined as the graph that has connected components of $G \setminus F$ as vertices; two such components $X$ and $Y$ being adjacent in $G / F$ if and only if some vertex from $X$ is
adjacent to a vertex from $Y$ in graph $G$. If $E=XY \in E(G/F)$ is an edge in graph $G/F$, then we denote by $\widehat{E}$ the set of edges of $G$ that have one end vertex in $X$ and the other end vertex in $Y$, i.e. $\widehat{E}=  \lbrace xy \in E(G)\,|\,x \in V(X), y \in V(Y) \rbrace $.

\section{Computing the Mostar index from the quotient graphs}
\label{sec:main}

We show in this section that the Mostar index of a weighted graph can be computed from the corresponding quotient graphs.

\noindent
Throughout the section, let $G$ be a connected graph and $\lbrace F_1, \ldots, F_k \rbrace$ a c-partition of the set $E(G)$. Moreover, the quotient graph $G/F_i$ will be shortly denoted as $G_i$ for any $i \in \lbrace 1, \ldots, k \rbrace$. In addition, we define the function $\ell_i: V(G) \rightarrow V(G_i)$ as follows: for any $u \in V(G)$, let $\ell_i(u)$ be the connected component $U$ of the graph $G \setminus F_i$ such that $u \in V(U)$. The next lemma was obtained in \cite{klavzar-2016}, but the proof can be also found in \cite{tratnik_grao}.

\begin{lemma} \cite{klavzar-2016,tratnik_grao} \label{distance}
If $u,v \in V(G)$ are two vertices, then 
$$d_G(u,v) = \sum_{i=1}^k d_{G_i}(\ell_i(u),\ell_i(v)).$$
\end{lemma}


\noindent
The following lemma is the key to our method. The main ideas of the proof can be found in \cite{li}, but for the sake of completeness we give the whole proof.

\begin{lemma} \label{pomoc}
If $e =uv \in F_i$, where $i \in \lbrace 1, \ldots, k \rbrace$, then $U=\ell_i(u)$ and $V=\ell_i(v)$ are adjacent vertices in $G_i$, i.e.\ $E=UV \in E(G_i)$. Moreover,
$$N_u(e|G) = \bigcup_{X \in N_{U}(E|G_i)} V(X),$$
$$N_v(e|G) = \bigcup_{X \in N_{V}(E|G_i)} V(X).$$
\end{lemma}

\begin{proof}
Obviously, for any $j \in \lbrace 1, \ldots, k \rbrace$, $j \neq i$, it holds $\ell_j(u) = \ell_j(v)$ and therefore $d_{G_j}(\ell_j(u),\ell_j(v))=0$. By Lemma \ref{distance} we now obtain
$$d_G(u,v)= \sum_{j=1}^k d_{G_j}(\ell_j(u),\ell_j(v)) = d_{G_i}(\ell_i(u),\ell_i(v)),$$
which implies $d_{G_i}(\ell_i(u),\ell_i(v))=1$. Hence, $U=\ell_i(u)$ and $V=\ell_i(v)$ are adjacent vertices in $G_i$ and we denote $E=UV$.
Next, let $z \in V(G)$ be an arbitrary vertex in $G$. Again, for any $j \in \lbrace 1, \ldots, k \rbrace$, $j \neq i$, it holds $d_{G_j}(\ell_j(z), \ell_j(u)) = d_{G_j}(\ell_j(z),\ell_j(v))$. Therefore, by Lemma \ref{distance} we have
\begin{eqnarray*}
d_G(z,u) - d_G(z,v) & = & \sum_{j=1}^k d_{G_j}(\ell_j(z),\ell_j(u)) - \sum_{j=1}^k d_{G_j}(\ell_j(z),\ell_j(v)) \\
& = & \sum_{j=1}^k \left( d_{G_j}(\ell_j(z),\ell_j(u)) - d_{G_j}(\ell_j(z),\ell_j(v)) \right) \\
& = & d_{G_i}(\ell_i(z),\ell_i(u)) - d_{G_i}(\ell_i(z),\ell_i(v)).
\end{eqnarray*}
\noindent
We can see from the obtained equality that $d_G(z,u)< d_G(z,v)$ if and only if $d_{G_i}(\ell_i(z),\ell_i(u)) $ $< d_{G_i}(\ell_i(z),\ell_i(v))$. Hence, it holds $z \in N_u(e|G)$ if and only if $\ell_i(z) \in N_{\ell_i(u)}(E|G_i) = N_{U}(E|G_i)$, which is equivalent to $z \in V(X)$ for some $X \in N_{U}(E|G_i)$. This proves the following equality:
$$N_u(e|G) = \bigcup_{X \in N_{U}(E|G_i)} V(X).$$
The remaining equality can be shown in the same way. \qed 
\end{proof}
\smallskip

\noindent
Let $w: V(G) \rightarrow \mathbb{R}_0^+$, $w': E(G) \rightarrow \mathbb{R}_0^+$ be given weights and $i \in \lbrace 1, \ldots,k \rbrace$. We define $\lambda_i : V(G_i) \rightarrow \mathbb{R}_0^+$ in the following way: for any $X \in V(G_i)$, let $\lambda_i(X) = \sum_{x \in V(X)}w(x)$. So $\lambda_i(X)$ is the sum of all the weights of vertices from $X$.

\noindent
Moreover, we define $\lambda_i' : E(G_i) \rightarrow \mathbb{R}_0^+$ as follows: for any $E=XY \in E(G_i)$, let $\lambda_i'(E) = \sum_{e \in \widehat{E}}w'(e)$. Therefore, $\lambda_i'(E)$ is the sum of weights of edges that have one end vertex in $X$ and the other end vertex in $Y$.

\noindent
The following lemma will be needed as well.

\begin{lemma} \label{pomoc2} If $e=uv \in F_i$, where $i \in \lbrace 1, \ldots, k \rbrace$, $U=\ell_i(u)$, $V=\ell_i(v)$, and $E=UV \in E(G_i)$, then
$$n_u(e|(G,w)) = n_U(E|(G_i,\lambda_i)),$$
$$n_v(e|(G,w)) = n_V(E|(G_i,\lambda_i)).$$
\end{lemma}

\begin{proof}
By Lemma \ref{pomoc} we calculate
\begin{eqnarray*}
n_u(e|(G,w)) & = & \sum_{x \in N_u(e|G)} w(x) \\
& = & \sum_{X \in N_U(E|G_i)}  \left( \sum_{x \in V(X)} w(x) \right) \\
& = & \sum_{X \in N_U(E|G_i)} \lambda_i(X) \\
& = & n_U(E|(G_i,\lambda_i)),
\end{eqnarray*}
which proves the first equality. The other equality can be shown in the same way. \qed
\end{proof}
\smallskip

\noindent
Now we can state the main theorem.

\begin{theorem}
\label{mostar} If $(G,w,w')$ is a weighted connected graph and $\lbrace F_1,\ldots, F_k \rbrace$ is a c-partition of the set $E(G)$, then
$$Mo(G,w,w')=\sum_{i=1}^k Mo(G_i,\lambda_i,\lambda_i').$$
\end{theorem}
\begin{proof}
Obviously, it holds $E(G) = \displaystyle\bigcup_{i=1}^k F_i$. Moreover, for all $i \in \lbrace 1,\ldots, k \rbrace$ we have (by Lemma \ref{pomoc}) $$F_i= \bigcup_{E \in E(G_i)} \widehat{E}.$$ 
In the rest of the proof, we will write just $Mo$ instead of $Mo(G,w,w')$. Therefore, one can compute
\begin{eqnarray*}
Mo & = & \sum_{e =uv \in E(G)} w'(e) \left| n_u(e|(G,w)) - n_v(e|(G,w)) \right| \\
 & = & \sum_{i=1}^k \left( \sum_{e =uv \in F_i} w'(e) \left| n_u(e|(G,w)) - n_v(e|(G,w)) \right| \right) \\
  & = & \sum_{i=1}^k \left( \sum_{E=UV \in E(G_i)} \left[ \sum_{e =uv \in \widehat{E}} w'(e) \left| n_u(e|(G,w)) - n_v(e|(G,w)) \right| \right] \right).
\end{eqnarray*}
If $E=UV$ is an edge in $G_i$ and $e=uv$ is an arbitrary edge from $\widehat{E}$, then by Lemma \ref{pomoc2} we have
$$\left| n_u(e|(G,w)) - n_v(e|(G,w)) \right| = \left| n_U(E|(G_i,\lambda_i)) - n_V(E|(G_i,\lambda_i)) \right|.$$
Finally,
\begin{eqnarray*}
 Mo & = & \sum_{i=1}^k \left( \sum_{E=UV \in E(G_i)} \left[ \sum_{e \in \widehat{E}} w'(e) \left| n_U(E|(G_i,\lambda_i)) - n_V(E|(G_i,\lambda_i)) \right| \right] \right)  \\
& = & \sum_{i=1}^k  \left( \sum_{E=UV \in E(G_i)} \left| n_U(E|(G_i,\lambda_i)) - n_V(E|(G_i,\lambda_i)) \right| \left[ \sum_{e  \in \widehat{E}} w'(e)  \right] \right) \\
& = & \sum_{i=1}^k \left( \sum_{E=UV \in E(G_i)} \lambda_i'(E) \left| n_U(E|(G_i,\lambda_i)) - n_V(E|(G_i,\lambda_i)) \right|  \right) \\
& = & \sum_{i=1}^k Mo(G_i,\lambda_i,\lambda_i'),
\end{eqnarray*}
which is what we wanted to prove. \qed
\end{proof}

\noindent
If $w(u)=1$ for any $u \in V(G)$ and $w'(e)=1$ for any $e \in E(G)$, then $Mo(G,w,w')=Mo(G)$, which leads to the next corollary.

\begin{corollary}
\label{posl_izreka}
If $G$ is a connected graph and $\lbrace F_1,\ldots, F_k \rbrace$ a c-partition of the set $E(G)$, then
$$Mo(G)=\sum_{i=1}^k Mo(G_i,\lambda_i,\lambda_i'),$$
where $\lambda_i: V(G/F_i) \rightarrow \mathbb{R}_0^+$, $\lambda_i': E(G/F_i) \rightarrow \mathbb{R}_0^+$ are defined as follows: $\lambda_i(X)$ is the number of vertices in the connected component $X$ of $G \setminus F_i$ and $\lambda_i'(XY)$ is the number of edges in the set $\widehat{XY}$ (the number of edges between $X$ and $Y$).
\end{corollary}

\section{The Mostar index of benzenoid systems}

In this section, we show how the obtained method can be used to efficiently calculate the Mostar index of a benzenoid system. Note that such a computation can be done even by hand.

\noindent
Let ${\cal H}$ be the hexagonal (graphite) lattice and let $Z$ be a cycle on it. A {\em benzenoid system} is the graph induced by the vertices and edges of ${\cal H}$, lying on $Z$ or in its interior. The benzenoid systems defined in this way are sometimes called \textit{simple} \cite{DGKZ-2002}. In the figures, we usually do not use dots to denote the vertices of benzenoid systems. For an example of a benzenoid system see Figure \ref{koronen} or Figure \ref{benzenoid}. More information on these molecular graphs can be found in \cite{gucy-89}.

\noindent
An \textit{elementary cut} of a benzenoid system $G$ is a line segment that starts at the center of a peripheral edge of a benzenoid system,
goes orthogonal to it and ends at the first next peripheral
edge of $G$. The main insight for our consideration
is that every $\Theta$-class of a benzenoid system
$G$ coincides with exactly one of its elementary cuts. Therefore, we can easily see that benzenoid systems are partial cubes \cite{klavzar-book}. As a consequence, by removing all the edges that correspond to an elementary cut of a benzenoid system, the obtained graph has exactly two connected components.

\noindent
The edge set of a benzenoid system $G$ can be naturally partitioned into sets $F_1, F_2$, and $F_3$ of edges of the same direction. Obviously, the partition $\lbrace F_1,F_2,F_3 \rbrace$ is a c-partition of the set $E(G)$. For $i \in \lbrace 1, 2, 3 \rbrace$, let $T_i = G/F_i$ be the corresponding quotient graph. It is well known that $T_1$, $T_2$, and $T_3$ are trees \cite{chepoi-1996}.

\noindent
Next, we define the weights $\lambda_i : V(T_i) \rightarrow \mathbb{R}_0^+$ and $\lambda_i' : E(T_i) \rightarrow \mathbb{R}_0^+$ as in Corollary \ref{posl_izreka}:
\begin{enumerate}
\item for $X \in V(T_i)$, let $\lambda_i(X)$ be the number of vertices in the component $X$ of $G \setminus F_i$;
\item for $E = XY \in E(T_i)$, let $\lambda_i'(E)$ be the number of edges between components $X$ and $Y$ (the number of edges in the set $\widehat{E}$).
\end{enumerate}

\noindent
By Corollary \ref{posl_izreka} we immediately arrive to the following proposition.

\begin{proposition} \label{ben_formula}
If $G$ is a benzenoid system, then
$$Mo(G) = Mo(T_1,\lambda_1,\lambda_1') + Mo(T_2,\lambda_2,\lambda_2') + Mo(T_3,\lambda_3,\lambda_3').$$
\end{proposition}

\noindent
The following lemma will be also needed. 
\begin{lemma} 
\label{lema_mo}
Suppose $(T,w,w')$ is a  weighted tree with $n$ vertices. Then the Mostar  index $Mo(T,w,w')$ can be computed in $O(n)$ time.
\end{lemma}

\begin{proof}
The proof is based on the standard BFS (breadth-first search) algorithm and is almost the same as the proof of Proposition 4.4 in \cite{cre-trat}. \qed
\end{proof}
\smallskip

\noindent
In \cite{chepoi-1996} it was shown that each quotient tree $(T_i,\lambda_i,\lambda_i')$, $i \in \lbrace 1,2,3 \rbrace$, can be computed in linear time with respect to the number of vertices in a benzenoid system. Therefore, by Lemma \ref{lema_mo} and Proposition \ref{ben_formula}, the Mostar index of a benzenoid system $G$ can be computed in linear time $O(|V(G)|)$. However, the Mostar index can be computed even faster, i.e.\ in sub-linear time. To show this, we recall the next lemma which was stated as Lemma 3.1 in \cite{cre-trat1}. The proof uses a special construction of weighted trees $(T_i,\lambda_i,\lambda_i')$, $i \in \lbrace 1,2,3 \rbrace$, that depends only on the boundary cycle.

\begin{lemma}\label{tree_O(z)} \cite{cre-trat1}
	If $G$ is a benzenoid system and $Z$ its boundary cycle, then each weighted tree $(T_i, \lambda_i, \lambda_i')$, $i \in \{1,2,3\}$, can be obtained in $O(|Z|)$ time.	
\end{lemma}

\noindent The main result of this section can now be stated.

\begin{theorem}
If $G$ is a benzenoid system with the boundary cycle $Z$, then the Mostar index $Mo(G)$ can be computed in $O(|Z|)$ time.
\end{theorem}

\begin{proof}
By Lemma \ref{tree_O(z)} it follows that the trees $(T_i,\lambda_i,\lambda_i')$, $i \in \lbrace 1,2,3 \rbrace$, can be obtained in $O(|Z|)$ time. Moreover, it is easy to see that $|Z|= 2|E(T_1)| + 2|E(T_2)| + 2|E(T_3)|$ since every elementary cut corresponds to exactly one edge in a quotient tree and goes through exactly two edges from $Z$. Therefore, by Lemma \ref{lema_mo} the Mostar index of each weighted tree can be computed in linear time with respect to $|Z|$. Finally, by using Proposition \ref{ben_formula} we obtain that the Mostar index $Mo(G)$ can be computed in $O(|Z|)$ time. \qed
\end{proof}

In the rest of the section, we apply Proposition \ref{ben_formula} to  some benzenoid systems. As the first example, we calculate the Mostar index for an infinite family of molecular graphs called \textit{coronenes}, which was already done in Theorem 14 of \cite{mostar}. However, we now show how the same result can be achieved by using our method. In particular, coronene $G_1$ is just a single hexagon, and $G_h$ is obtained from $G_{h-1}$ by adding a ring of hexagons around it. Coronene $G_3$ is depicted in Figure \ref{koronen}.

\begin{figure}[h!] 
\begin{center}
\includegraphics[scale=0.6,trim=0cm 0cm 0cm 0cm]{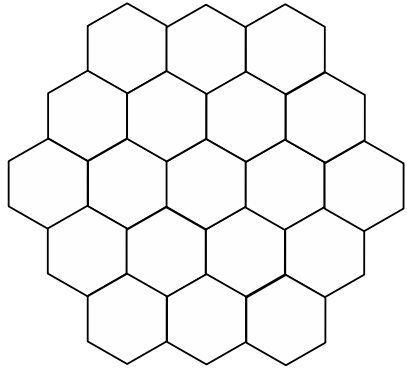}
\end{center}
\caption{\label{koronen} Coronene $G_3$.}
\end{figure}

Firstly, we have to determine the weighted quotient trees $(T_i,\lambda_i,\lambda_i')$ for any $i \in \lbrace 1,2,3 \rbrace$. Because of the symmetry, all the weighted trees are equal. Let $F_1$ be the set of all the vertical edges in $G_h$, $h \geq 1$. For the graph $G_3$ these edges and the corresponding elementary cuts are shown in Figure \ref{prerezi} (a). Moreover, the weighted quotient tree $(T_1,\lambda_1,\lambda_1')$ can be seen in Figure \ref{prerezi} (b).

\begin{figure}[h!] 
\begin{center}
\includegraphics[scale=0.6,trim=0cm 0cm 0cm 0cm]{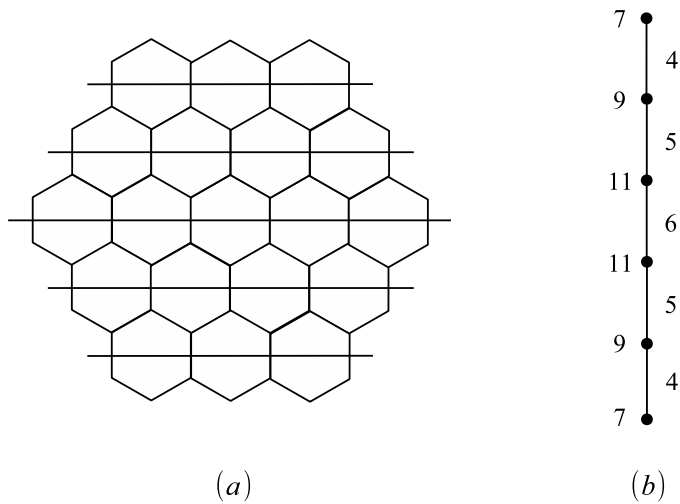}
\end{center}
\caption{\label{prerezi} (a) Horizontal elementary cuts for coronene $G_3$ and (b) the weighted tree $(T_1,\lambda_1,\lambda_1')$.}
\end{figure}

\noindent
However, we can easily generalize the above example to coronene $G_h$ and obtain that the quotient tree $T_1$ is isomorphic to the path on $2h$ vertices. Moreover, the weighted tree $(T_1,\lambda_1,\lambda_1')$ is depicted in Figure \ref{kvocient}.

\begin{figure}[h!] 
\begin{center}
\includegraphics[scale=0.6,trim=0cm 0cm 0cm 0cm]{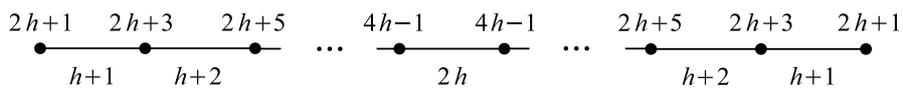}
\end{center}
\caption{\label{kvocient} Weighted quotient tree $(T_1,\lambda_1,\lambda_1')$ for graph $G_h$.}
\end{figure}

\noindent
Therefore, it is easy to compute
\begin{eqnarray*}
Mo(T_1,\lambda_1,\lambda_1') & = & 2 \sum_{i = h+1}^{2h-1} \left( 2i \sum_{j=i}^{2h-1}(2j+1)  \right) \\
& = & 9h^4 - 6h^3 - 3h^2.
\end{eqnarray*}
\noindent
Finally, by Proposition \ref{ben_formula} we have
\begin{eqnarray*}
Mo(G_h) & = & Mo(T_1,\lambda_1,\lambda_1') + Mo(T_2,\lambda_2,\lambda_2') + Mo(T_3,\lambda_3,\lambda_3') \\
& = & 3 Mo(T_1,\lambda_1,\lambda_1') \\
& = & 27h^4 - 18h^3 - 9h^2,
\end{eqnarray*}
which coincides with the result from \cite{mostar}.
\smallskip

As the second example, we compute the Mostar index for a branched benzenoid system $G$ from Figure \ref{benzenoid}.

\begin{figure}[h!] 
\begin{center}
\includegraphics[scale=0.6,trim=0cm 0cm 0cm 0cm]{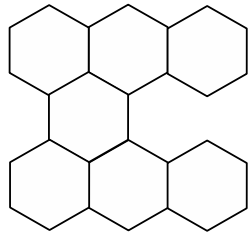}
\end{center}
\caption{\label{benzenoid} Benzenoid system $G$.}
\end{figure}

Again, let $F_1$ be the set of all vertical edges of $G$ and let $F_2,F_3$ be the edges in the other two directions. Then, the weighted quotient trees are shown in Figure \ref{trees}.

\begin{figure}[h!] 
\begin{center}
\includegraphics[scale=0.55,trim=0cm 0cm 0cm 0cm]{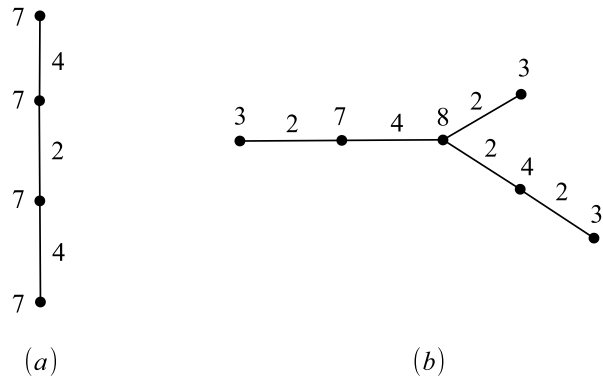}
\end{center}
\caption{\label{trees} Weighted quotient trees (a) $(T_1,\lambda_1,\lambda_1')$ and (b) $(T_2,\lambda_2,\lambda_2') = (T_3,\lambda_3,\lambda_3')$ for graph $G$.}
\end{figure}

\noindent
Hence, we can compute
\begin{eqnarray*}
Mo(T_1,\lambda_1,\lambda_1') & = & 4 \cdot 14 + 4 \cdot 14 = 112, \\
Mo(T_2,\lambda_2,\lambda_2') =Mo(T_2,\lambda_2,\lambda_2')
& = & 2 \cdot 22 + 4 \cdot 8 + 2 \cdot 22 + 2 \cdot 14 + 2 \cdot 22 = 192.
\end{eqnarray*}

\noindent
Finally, by Proposition \ref{ben_formula} one can calculate
\begin{eqnarray*}
Mo(G) & = & Mo(T_1,\lambda_1,\lambda_1') + Mo(T_2,\lambda_2,\lambda_2') + Mo(T_3,\lambda_3,\lambda_3') \\
& = & 112 + 192 + 192 \\
& = & 496.
\end{eqnarray*}

We should mention that by using our method,  similar results can be deduced for various families of interesting molecular graphs and networks, in particular for phenylenes \cite{brez-trat} and $C_4C_8$ systems \cite{cre-trat}. In these cases, the Mostar index can be computed in terms of four (instead of three) weighted quotient trees.

\section{A fullerene patch} A \textit{fullerene} $F$ is a 3-regular plane graph with only pentagonal and hexagonal faces. If $C$ is an elementary cycle in $F$, then $C$ partitions
the plane into two open regions. A \textit{patch} of $F$ is defined as the graph obtained from $F$ by deleting all vertices (and edges) in the interior of one of the two regions \cite{graver}. 

\noindent
In this section, we apply our method to compute the Mostar index of a patch that is obtained from the well known buckminsterfullerene $C_{60}$. Therefore, let $G$ be the graph shown in Figure \ref{nanocone} (a). However, graph $G$ also belongs to another family of important chemical structures called \textit{nanocones}. Generally speaking, nanocones are planar graphs where the inner faces are mostly hexagons, but there can be also some non-hexagonal inner faces, most commonly pentagons.

Firstly, we have to determine the $\Theta^*$-classes of $G$, which are denoted by $E_1,E_2,E_3,E_4$, $E_5,E_6$ and shown in Figure \ref{nanocone} (b). Note that the $\Theta^*$-classes of $G$ were already obtained in \cite{li}, where the revised edge-Szeged index was computed for this graph. However, here we will use another partition of the set $E(G)$ and therefore obtain different quotient graphs. It turns out that for graph $G$ relation $\Theta$ is not transitive and hence, $G$ is not a partial cube. 

\begin{figure}[h!] 
\begin{center}
\includegraphics[scale=0.6,trim=0cm 0cm 0cm 0cm]{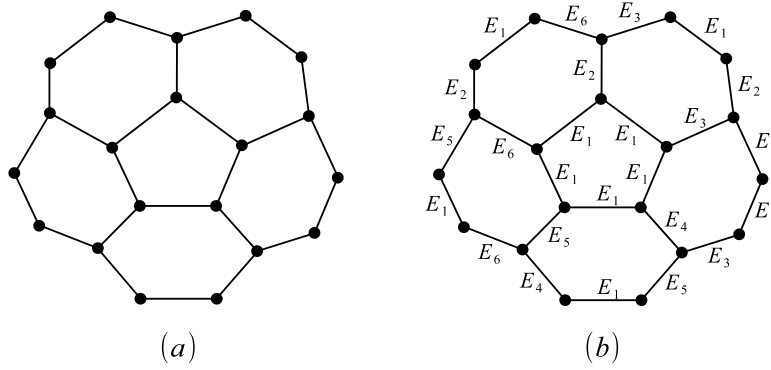}
\end{center}
\caption{\label{nanocone} (a) Graph $G$ and (b) the $\Theta^*$-classes of $G$.}
\end{figure}

\noindent
Let $F_1=E_1$ and $F_2 = E_2 \cup E_3 \cup E_4 \cup E_5 \cup E_6$. Obviously, $\lbrace F_1,F_2 \rbrace$ is a c-partition of $E(G)$. Next, the weighted quotient graphs $(G_1,\lambda_1,\lambda_1')$ and $(G_2,\lambda_2,\lambda_2')$ can be easily determined, see Figure \ref{kvo_grafa}. As in Section 3, the graph $G_i$ denotes the quotient graph $G/F_i$ for $i \in \lbrace 1,2 \rbrace$. Moreover, the weights are calculated as in Corollary \ref{posl_izreka}.

\begin{figure}[h!] 
\begin{center}
\includegraphics[scale=0.6,trim=0cm 0cm 0cm 0cm]{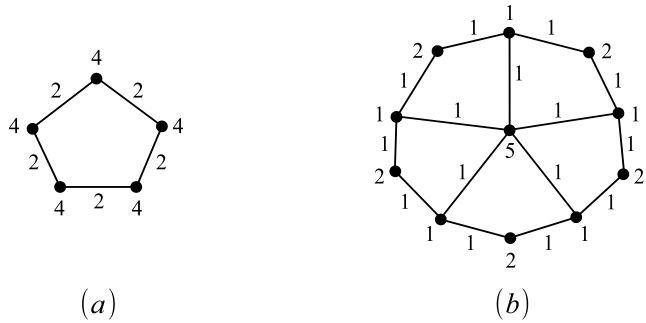}
\end{center}
\caption{\label{kvo_grafa} Weighted quotient graphs: (a) $(G_1,\lambda_1,\lambda_1')$ and (b) $(G_2,\lambda_2,\lambda_2')$.}
\end{figure}

\noindent
It is obvious that $Mo(G_1,\lambda_1,\lambda_1')=0$. Moreover, to calculate $Mo(G_2,\lambda_2,\lambda_2')$, we consider just two types of edges, i.e.\ the five edges that have the central vertex for an end-vertex and the remaining ten edges. Therefore, by Corollary \ref{posl_izreka} the calculation of the Mostar index becomes trivial:
\begin{eqnarray*}
Mo(G) & = & Mo(G_1,\lambda_1,\lambda_1') + Mo(G_2,\lambda_2,\lambda_2') \\
& = & Mo(G_2,\lambda_2,\lambda_2') \\
& = & 5 \cdot 1 \cdot |15 - 5| + 10 \cdot 1 \cdot |15-5| \\ 
& = & 150.
\end{eqnarray*}

\section*{Acknowledgement}
\noindent
The author Niko Tratnik acknowledge the financial support from the Slovenian Research Agency (research core funding No. P1-0297 and J1-9109).


\baselineskip=16pt

\begin{thebibliography}{99}
  
 \bibitem{arock1} M. Arockiaraj, S. Klav\v zar, S. Mushtaq, K. Balasubramanian, Distance-based topological indices of nanosheets, nanotubes and nanotori of   $\textrm{SiO}_2$, J. Math. Chem. 57 (2019) 343--369.

\bibitem{brez-trat} S. Brezovnik, N. Tratnik, New methods for calculating the degree distance and the Gutman index, MATCH Commun. Math. Comput. Chem. 82 (2019) 111--132.
  

  
  \bibitem{chepoi-1996}
  V.~Chepoi, 
  On distances in benzenoid systems, J. Chem. Inf. Comput. Sci. {36} (1996) 1169--1172.
  
  \bibitem{cre-trat} M.~\v Crepnjak, N.~Tratnik, The Szeged index and the Wiener index of partial cubes with applications to chemical graphs, Appl. Math. Comput. 309 (2017) 324--333.

\bibitem{cre-trat1} M.~\v Crepnjak, N.~Tratnik, The edge-Wiener index, the Szeged indices and the PI index of benzenoid systems in sub-linear time, MATCH Commun. Math. Comput. Chem. 78 (2017) 675--688.

  
  \bibitem{dehmer} M. Dehmer, F. Emmert-Streib (Eds.), Quantitative Graph Theory: Mathematical Foundations and Applications, CRC Press, Taylor \& Francis Group, Boca Raton, 2015.


	\bibitem{DGKZ-2002}
	A.~A.~Dobrynin, I.~Gutman, S.~Klav\v zar, P.~\v Zigert,
	Wiener index of hexagonal systems,
	Acta Appl. Math. 72 (2002) 247--294.
	
\bibitem{mostar} T. Do\v sli\' c, I. Martinjak, R. \v Skrekovski, S. Tipuri\' c Spu\v zevi\' c, I. Zubac, Mostar index, J. Math. Chem. 56 (2018) 2995--3013.

\bibitem{estrada} E. Estrada, The Structure of Complex Networks, Oxford University Press, New York, 2011.

\bibitem{graver} J. E. Graver, C. M. Graves, Fullerene patches I, Ars Math. Contemp. 3 (2010) 109--120.

\bibitem{gut_sz} I. Gutman, A formula for the Wiener number of trees and its extension to graphs containing cycles, Graph Theory Notes N. Y. 27 (1994) 9--15. 

\bibitem{gucy-89}
I.  Gutman, S.~J. Cyvin, Introduction to the Theory of Benzenoid Hydrocarbons, Springer-Verlag, Berlin, 1989.
 
  \bibitem{ji} S. Ji, M. Liu, J. Wu, A lower bound of revised Szeged index of bicyclic graphs, Appl. Math. Comput. 316 (2018) 480--487.
  
  \bibitem{klavzar-book}
R. Hammack, W. Imrich, S. Klav\v{z}ar, Handbook of Product Graphs, Second Edition, CRC Press, Taylor \& Francis Group, Boca Raton, 2011.


\bibitem{klavzar_li} S. Klav\v zar, S. Li, H. Zhang, On the difference between the (revised) Szeged index and the Wiener index of cacti, Discrete Appl. Math. 247 (2018) 77--89.

  



\bibitem{klavzar-2015}
  S.~Klav\v zar, M.~J.~Nadjafi-Arani, Cut method: update on recent developments and equivalence of independent approaches,
  {Curr.\ Org.\ Chem}.\ 19 (2015) 348--358.
  
  \bibitem{klavzar-2016} S.~Klav\v zar, M.~J.~Nadjafi-Arani, Wiener index in weighted graphs via unification of $\Theta^{*}$-classes, European J. Combin. 36 (2014) 71--76.   

  


  \bibitem{knor1} M. Knor, R. \v Skrekovski, A. Tepeh, Mathematical aspects of Wiener index, Ars Math. Contemp. 11 (2016) 327--352.

\bibitem{li} X. Li, M. Zhang, A Note on the computation of revised (edge-)Szeged index in terms of canonical isometric embedding, MATCH Commun. Math. Comput. Chem. 81 (2019) 149--162.

\bibitem{todeschini} R. Todeschini, V. Consonni, Handbook of Molecular Descriptors, WILEY-VCH, Weinheim, 2002. 



  \bibitem{tratnik_grao} N. Tratnik, The Graovac-Pisanski index of zig-zag tubulenes and the generalized cut method, J. Math. Chem. 55 (2017) 1622--1637.

\bibitem{wiener} H. Wiener, Structural determination of paraffin boiling points, J. Amer. Chem. Soc. 69 (1947) 17--20.
  
    
\end{thebibliography}
\end{document}